\documentclass[11pt]{article}
\usepackage[english]{babel}
\usepackage{amsmath,amsthm}
\usepackage{amsfonts}
\usepackage{mathrsfs}
\usepackage{color}
\usepackage{bbm}
\usepackage[titletoc]{appendix}
\newtheorem{thm}{Theorem}[section]
\newtheorem{cor}[thm]{Corollary}
\newtheorem{lem}[thm]{Lemma}
\newtheorem{prop}[thm]{Proposition}
\newtheorem{ass}[thm]{Assumption}
\theoremstyle{definition}
\newtheorem{defn}[thm]{Definition}
\theoremstyle{remark}
\newtheorem{rem}[thm]{Remark}

\numberwithin{equation}{section}
\begin{document}
\def\Pro{{\mathbb{P}}}
\def\E{{\mathbb{E}}}
\def\e{{\varepsilon}}
\def\ds{{\displaystyle}}
\def\nat{{\mathbb{N}}}
\def\Dom{{\textnormal{Dom}}}

\title{Markov processes with spatial delay: path space characterization, occupation time and properties}%
\author{M. Salins and K. Spiliopoulos\footnote{
Boston University, Department of Mathematics and Statistics, 111 Cummington Mall, Boston, MA, 02215. The present work was partially supported by NSF grant DMS-1312124 and during revisions of this article by NSF CAREER award DMS 1550918.}}
\maketitle
\begin{abstract}
In this paper, we study one dimensional Markov processes with spatial delay. Since the seminal work of Feller, we know that virtually any one dimensional, strong, homogeneous, continuous Markov process can be uniquely characterized via its infinitesimal generator and the generator's domain of definition. Unlike standard diffusions like Brownian motion, processes with spatial delay spend positive time at a single point of space. Interestingly, the set of times that a delay process spends at its delay point is nowhere dense and forms a positive measure Cantor set. The domain of definition of the generator has restrictions involving second derivatives. In this article we provide a pathwise characterization for processes with delay in terms of an SDE and an occupation time formula involving the symmetric local time. This characterization provides an explicit Doob-Meyer decomposition, demonstrating that such processes are semi-martingales and that all of stochastic calculus including It\^{o} formula and Girsanov formula applies. We also establish an occupation time formula linking the time that the process spends at a delay point with its symmetric local time there. A physical example of a stochastic dynamical system with delay is lastly presented and analyzed.
\end{abstract}




\section{Introduction}

In this paper we study continuous strong Markov processes in dimension one
that may have points with reflection, partial reflection (e.g., skew diffusion) and points with spatial delay (sticky points). Since the seminal work of Feller \cite{Feller}, see also \cite{M1,ItoMcKean,RevuzYor}, it is known that, under some minimal regularity conditions,  every one dimensional, continuous, homogeneous strong Markov process can be uniquely characterized by its infinitesimal generator and the generator's domain of definition, and vice versa. The generator and its domain of definition take a specific form and are usually denoted by $D_{v}D_{u}$ where $v,u$ are strictly increasing functions, $v$ is right-continuous and $u$ is continuous, whereas $D_{\cdot}$ is an appropriate differential operator, see Section \ref{S:DvDu}.

The Feller characterization of one-dimensional Markov processes is based on the so-called scale function and speed measure and includes as special cases standard It\^{o} diffusions, diffusions with reflection, diffusions that have asymmetric probabilities of exiting from left or right of a small neighborhood of a given point (points of partial reflection, e.g., skew diffusion) and it also includes processes that may  have spatial delay at certain points (e.g., sticky points). By delay, we mean that the process spends positive time at a particular point of space (to be made precise below). However, in the general case, one can describe such Markov processes only through their generator and no pathwise description exists in the literature (that we are aware of)  so far except for the cases of standard diffusions, diffusions with (partial) reflection and sticky Brownian motion, see for example \cite{KS,RevuzYor,EngelbertPeskir2014,EngelbertSchmidt1991,Schmidt1989}.

If $b: \mathbb{R} \to \mathbb{R}$, and $\sigma: \mathbb{R} \to \mathbb{R}$ are Lipschitz continuous, then it is well known that the solution to the stochastic differential equation (SDE)
\[dX(t) = b(X(t))dt + \sigma(X(t))dW(t), \ \ X(0)=x\]
behaves locally like a Brownian motion. In particular, if
 \[\tau_\delta = \inf\{t>0:|X(t) - x|>\delta\},\]
  then
\[\lim_{\delta \downarrow 0} \Pro_x(X(\tau_\delta)= x+\delta) = \lim_{\delta \downarrow 0} \Pro_x(X(\tau_\delta)= x-\delta) =\frac{1}{2}\]
and $\E(\tau^\delta) = O(\delta^2)$ as $\delta \to 0$. That is,  the solution to any classical SDE will be equally likely to go left or right, and will spend a quadratic amount of time in an arbitrarily small neighborhood of any point. A Markov process with (partial) reflection (i.e. asymmetry) at the point $x$ will satisfy
\[\lim_{\delta \downarrow 0} \Pro_x(X(\tau_\delta)) = p_+, \ \lim_{\delta \downarrow 0} \Pro_x(X(\tau_\delta) = -\delta) = p_-=1-p_+.\]
If $p_+=1$, then the process is totally reflected in the positive direction at $x$.
A process with $\alpha$ delay at the point $x$ will satisfy
\[\lim_{\delta \downarrow 0} \frac{1}{\delta}\E_{x}(\tau_\delta) = \alpha.\]
This means that delayed processes spend a lot more time in a $\delta$-neighborhood of $x$ than a classical diffusion would.

Let $E =\{x_1,x_2,...\} \subset \mathbb{R}$ be a finite or countable set of points (notice that it may even have accumulation points). Consider a process $X(t)$ that behaves like $dX(t) = b(X(t)) dt + \sigma(X(t))dW(t)$ away from the set $E$ and experiences partial reflection with parameters $p^i_++p^i_-=1$ and delay with parameter $\alpha_i\geq 0$ at the point $x_i$. The reflection and delay can be characterized in terms of the boundary conditions for the infinitesimal generator of the process,
\[\mathscr{L}f(x) = \frac{1}{2}\sigma^2(x)f''(x) + b(x) f'(x)\]
\begin{align}
\textnormal{Dom}(\mathscr{L})&= \left\{f \in C(\mathbb{R})\cap C^2(\mathbb{R} \setminus E): \mathscr{L}f \in C(\mathbb{R}),\right.\nonumber\\
&\qquad  \left. p^i_+f'(x_i+)-p^i_-f'(x_i-) = \alpha_i \mathscr{L}f(x_i),x_i \in E\right\}\nonumber
\end{align}

That is, the domain of definition of $\mathscr{L}$ consists of functions $f$ which are are twice continuously differentiable on $\mathbb{R}\setminus E$. The first derivatives of the functions may be discontinuous at the points in $E$, but the first derivatives must have left and right limits denoted as $f'(x_i-)$ and $f'(x_i+)$ respectively. Despite that lack of continuity of the first and second derivatives, $\mathscr{L}f$ is continuous. Lastly, such $f$ satisfy the boundary conditions above.

It has been shown by many authors, see for example the survey work by Lejay \cite{Lejay}, that a process with partial reflection, say at the point zero, but with no delay (i.e., $\alpha=0$) solves the SDE
\[dX(t) = b(X(t))dt + \sigma(X(t))dW(t) + (p_+-p_-)L^X(dt,0),\]
where $L^X(t,0)$ is the symmetric local time of $X$ at 0 (see Section \ref{S:local time} for definition and properties). Notice that when $p_+=p_- = 1/2$, the process solves a standard SDE.

With the exception of the recent results on sticky Brownian motion \cite{EngelbertPeskir2014}, see also \cite{Bass2014}, no such pathwise representation is known in the case of  one-dimensional processes that may have spatial delay at certain points. In this paper, we address this question. We demonstrate that all one-dimensional diffusions with spatial delay can be written as a time-changed version of a process with no points of delay. Then, we show that these delayed processes satisfy a surprisingly simple stochastic differential equation. We derive a pathwise representation for general one-dimensional Markov processes that may have points of both partial reflection and of delay, see Theorem \ref{T:SDE-representation}. This result allows us to better understand the role of delay in the evolution of the process, the interaction with local times and allow us to develop stochastic calculus and Meyer-Tanaka formula, see Theorem \ref{T:Meyer-Tanaka}.

In particular, any process with partial reflection $p_++p_-=1$ and delay $\alpha\geq0$ at the point $0$ solves the SDE
\begin{align}
dX(t) &= b(X(t))\mathbbm{1}_{\{X(t)\not = 0\}}dt + \sigma(X(t))\mathbbm{1}_{\{X(t)\not =0\}}dW(t) + (p_+-p_-)L^X(dt,0)\nonumber\\
    X(0) &= x, \text{ and }    \int_0^t \mathbbm{1}_{\{X(s) = 0\}}ds = \alpha L^X(t,0).\nonumber
\end{align}

The above representation has many interesting consequences. Unlike a standard diffusion, a delayed process spends positive time at its delay points. This is of course related to the so-called slowly reflecting boundary points, see for example Chapter VII, Section 3 of \cite{RevuzYor}, and to sticky Brownian motion, see \cite{Bass2014,EngelbertPeskir2014}. The occupation formula $\int_0^t \mathbbm{1}_{\{X(s) = 0\}}ds = \alpha L^X(t,0)$ holds, which characterizes both the occupation time of the process $X$ at zero and its local time. The fact that the occupation time is positive is interesting because the set $\{t>0:X(t)=0\}$ contains no intervals and is actually nowhere dense. Notice that while the local time for a standard diffusion is unbounded, the local time of a delayed process is bounded on finite time intervals ($L^X(t,0) \leq \frac{t}{\alpha}$). Furthermore, this representation shows that these delayed processes are semimartingales, and therefore It\^{o} formula and Girsanov formula for such delayed processes follow immediately from well-known classical theories, see Lemma \ref{lem:explicit-martingale-prob} and Corollary \ref{C:Girsanov}. Moreover, in the case of continuous $b$ and $\sigma$ a more compact formula is available, see Corollary \ref{cor:cont-coeff}.

The rest of the paper is organized as follows. In Section \ref{S:One-d-delay-reflection} we explain the problem in more detail and we provide a way to transform a process without delay to a process with delay. In Section \ref{S:Occu}, we relate the occupation time and the local time of diffusion processes with delay and establish the path space representation as a solution to an SDE.  In Section \ref{S:DvDu} we make connections of our results with those of Feller \cite{Feller} and Volkonskii \cite{Volkonskii1,Volkonskii2}. In Section \ref{S:Example} we consider a specific physical example, in particular Wiener process with reflection in narrow tubes, which in the limit as the tube becomes narrower may converge to a process with delay. We conclude with Appendix \ref{S:local time}, where we review some of the properties of symmetric local time, and with Appendix \ref{S:BrownianMotionDelay_Properties} where we recall and prove for completeness distributional properties, such as the characteristic function, of Brownian motion with a delay point of arbitrary delay.

\section{One dimensional diffusion with one point of delay or reflection}\label{S:One-d-delay-reflection}
In this section we represent processes with delay as time-changed versions of processes without delay. For simplicity, in this section we consider a process with only one point of partial reflection or delay whose differential generator is
\begin{equation} \label{eq:infites-gen}
  \mathscr{L}f(x) = \frac{1}{2}\sigma^2(x) f''(x) + b(x) f'(x)
\end{equation}
with domain of definition
\begin{equation} \label{eq:boundary-conds}
  \textnormal{Dom}(\mathscr{L}) = \left\{ \begin{array}{l} \ds{ f \in C^2(\mathbb{R} \setminus \{0\}) \cap C(\mathbb{R}): \mathscr{L}f \in C(\mathbb{R}),} \\
   \ds{p_+ f'(0+)-p_-f'(0-) = \alpha \mathscr{L}f(0)} \end{array} \right\}
\end{equation}
\begin{ass} \label{ass:b-and-sig-regularity}
  In the above equations, $b$ and $\sigma$ are uniformly Lipschitz continuous on $(0,+\infty)$ and $(-\infty,0)$. Both $b$ and $\sigma$ may have a jump discontinuity at the point $x=0$.
 There exists $c>0$ such that $\sigma^2(x)\geq c$, and $ p_+,p_-, \alpha \geq 0$ and $p_++p_-=1$.
\end{ass}
By the classical results of \cite{Feller}, (\ref{eq:infites-gen})-(\ref{eq:boundary-conds}) define in a unique way a strong Markov, homogeneous, continuous process $X(t)$.
We will build a stochastic process whose infinitesimal generator is $\mathscr{L}$. First, consider the solution to the SDE
\begin{equation} \label{eq:SDE-Y}
  dY(t) = b(Y(t)) dt + \sigma(Y(t))dW(t) + (p_+-p_-)L^Y(dt,0),
\end{equation}
where $L^Y(t,0)$ is the symmetric local time of $Y$ at $0$. See Appendix \ref{S:local time} for the definition and properties.

Processes like $Y$ satisfying (\ref{eq:SDE-Y}) have been well studied in the literature, see for example  \cite{FreidlinSheu2000,Lejay} among others.
The process satisfying (\ref{eq:SDE-Y}) almost satisfies the boundary conditions \eqref{eq:boundary-conds} except that it does not exhibit the $\alpha$ delay at $0$. We build the delay by defining the random time change
\begin{equation} \label{eq:r-def}
  r(t) = t + \alpha L^Y(t,0)
\end{equation}
and the process
\begin{equation} \label{eq:X-def-general}
  X(t) = Y(r^{-1}(t))
\end{equation}
where $r^{-1}$ is the functional inverse of the strictly increasing function $r(t)$.

Next we characterize the time change $r^{-1}(t)$ in terms of the local time $L^X(t,0)$.
\begin{lem}
  The time change has the representation
  \begin{equation} \label{eq:r-inv}
    r^{-1}(t) = t - \alpha L^X(t,0).
  \end{equation}
\end{lem}
\begin{proof}
  From the definition of local time, Definition \ref{def:local_time}, we have with probability one
  \[L^X(t,0) = L^Y(r^{-1}(t),0).\]

Due to the fact that $r^{-1}(t)$ is continuous, we get that it is in synchronization with the process $Y(t)$. In other words, $Y(t)$ is constant on time intervals $[r^{-1}(t-),r^{-1}(t)]$ almost surely. Then, Theorem 3.1 of \cite{Kobayashi} implies that for every $t\geq 0$, we have with probability one
\[
\int_{0}^{r^{-1}(t)}\sigma(Y(s))dY(s)=\int_{0}^{t}\sigma(Y(r^{-1}(s)))dY(r^{-1}(s))= \int_{0}^{t}\sigma(X(s))dX(s).
\]

Notice that by definition
  \begin{align*}
    &L^Y(r^{-1}(t),0) = |Y(r^{-1}(t))| - |Y(0)| - \int_0^{r^{-1}(t)}\textnormal{sign}(Y(s))dY(s) \\
    &= |X(t)| - |X(0)| - \int_0^t \textnormal{sign}(X(s))dX(s) = L^X(t,0).
  \end{align*}

Then from the definition of $r$ in \eqref{eq:r-def}
  \[t = r^{-1}(t) + \alpha L^Y(r^{-1}(t),0).\]

  The result follows by combining these two last results.
\end{proof}

\begin{thm}\label{T:GeneralTimeChangedMP_1}
  The process $X(t)$ defined in \eqref{eq:X-def-general} is a Markov process with infinitesimal generator \eqref{eq:infites-gen}-\eqref{eq:boundary-conds}.
\end{thm}
\begin{proof}

    First, we prove that $X$ is a Markov process. This is a consequence of the fact that $r^{-1}(s)$ is a stopping time for any $s>0$ with respect to the natural filtration $\mathcal{F}^Y_T$ associated with the Markov process $Y$. Let $t,T>0$. Notice that the set
    \[\{r^{-1}(t) \leq T\} = \{t \leq T + \alpha L^Y(T,0)\} \in \mathcal{F}^Y_T. \]
    Therefore, we can define $\mathcal{F}^X_t = \mathcal{F}^Y_{r^{-1}(t)}$.
    By the strong Markov property for $Y$, for any measurable set $I\subset \mathbb{R}$, and $0<s<t$,
    \begin{align*}
    &\Pro(X(t)\in I | \mathcal{F}^X_s) = \Pro(Y(r^{-1}(t))\in I|\mathcal{F}^Y_{r^{-1}(s)})\\
     &= \Pro(Y(r^{-1}(t))\in I|Y(r^{-1}(s))) = \Pro(X(t)\in I|X(s)).
    \end{align*}

   To prove that the generator of $X$ is given by \eqref{eq:infites-gen}-\eqref{eq:boundary-conds}, it is enough to check what happens at $x=0$. If $X(0) \not =0$, then $X$ behaves locally like a standard diffusion with generator $\mathscr{L}$ and there is no delay until the first time that $X$ hits $0$ because $r^{-1}(s) =s$ (and therefore $X(s) = Y(s)$).
  Assume $X(0)=0$ and define the stopping times
  \[\tau^\delta_X = \inf\{t>0: |X(t)|\geq \delta\} \text{ and } \tau^\delta_Y = \inf\{t>0: |Y(t)|\geq \delta \}.\]
  Observe that
  \[\tau^\delta_X = r(\tau^\delta_Y)\]
  and
  \[\E \tau^\delta_X = \E \tau^\delta_Y + \alpha \E L^Y(\tau^\delta_Y,0) = o(\delta) + \alpha \delta .\]
  The above line follows from the fact that by Definition \ref{def:local_time} $\E L^Y(\tau^\delta_Y,0) = \E|Y(\tau^\delta_Y)| = \delta$, and the fact that $\E(\tau^\delta_Y) = O(\delta^2)$ because $Y$ is a Markov process without delay.

  Now, on the one hand, by Taylor formula, we have for $\delta$ sufficiently small
  \[\E f(X(\tau^\delta_X)) -f(0) = f'(0+) \delta \Pro(X(\tau^\delta_X) = \delta) -f'(0-)\delta \Pro(X(\tau^\delta_X)=-\delta)+o(\delta).\]

  On the other hand, by Dynkin's formula, for $\delta$ small enough we have
  \[\E f(X(\tau^\delta_X)) - f(0) = \E \int_0^{\tau^\delta_X} \mathscr{L}f(X(s))ds = \alpha \delta\mathscr{L}f(0) + o(\delta).\]

  Comparing the above two formulas, dividing by $\delta$ and taking $\delta\downarrow 0$, we conclude that the appropriate boundary condition is
  \[p_+ f'(0+) - p_- f'(0-) = \alpha \mathscr{L}f(0),\]
  as desired. This concludes the proof given the uniqueness results due to Feller \cite{Feller}.
  \end{proof}

\section{Occupation time at a delay point and an SDE representation of processes with spatial delay} \label{S:Occu}
 It is well known that the zero set of Brownian motion has zero Lebesgue measure. That is, for any $T>0$,
 \begin{equation}
   \int_0^T \mathbbm{1}_{\{W(s) = 0\}} ds = 0 \text{ with probability 1. }
 \end{equation}
 The zero set of Brownian motion $Z_0 = \{t\geq 0 : W(t) =0\}$ is topologically a Cantor set with probability one. That is, $Z_0$ is a closed nowhere dense set that is its own boundary. A one-dimensional diffusion with spatial delay, however, will spend positive time at its delay points. Its occupation time set is still closed and nowhere dense, but it has positive measure. In this sense, the set of occupation times at a point of delay is topologically a so-called generalized Cantor set (see \cite{Royden} section 2.7). 

 \begin{thm} \label{thm:occupation-time}
    Consider a diffusion $X(t)$ with only one point of delay and/or (partial) reflection at $0$ whose infinitesimal generator is \eqref{eq:infites-gen} with boundary conditions \eqref{eq:boundary-conds}.
  Assume that $\sigma^2(x) \geq c>0$.
   For any $t>0$, the occupation time of the process $X(t)$ at its delay point is
   \begin{equation*} \label{eq:occupation-time}
     \int_0^t \mathbbm{1}_{\{X(s)=0\}} ds = \alpha L^X(t,0).
   \end{equation*}
 \end{thm}
 \begin{proof}
   By Theorem \ref{T:GeneralTimeChangedMP_1}, process $X(t)$ can be expressed as a time changed version of the SDE
 \begin{equation*}
   dY(t) = b(X(t)) dt + \sigma(Y(t))dW(t) + (p_+-p_-)L^Y(dt,0)
 \end{equation*}
 where $r(t) = t+\alpha L^Y(t,0)$ and $X(t) = Y(r^{-1}(t))$.
 Because $Y(t)$ has no points of delay, it spends almost no time at the point $0$. Therefore, for any $t>0$,
 \begin{equation*} \label{eq:undelayed-occu-time}
   \int_0^t \mathbbm{1}_{\{Y(s) =0\}} ds = 0.
 \end{equation*}
 By substituting $r^{-1}(t)$ for $t$ in the above formula, we observe that
 \begin{equation} \label{eq:occu=0}
   \int_0^{r^{-1}(t)} \mathbbm{1}_{\{Y(s) = 0\}} ds = 0.
 \end{equation}
 Then, because $r^{-1}(t)$ has finite variation (it is an increasing function), we can write the left-hand side as the Lebesgue-Steiltjes integral
 \begin{equation} \label{eq:Steiltjes}
   \int_0^{r^{-1}(t)} \mathbbm{1}_{\{Y(s) = 0\}}ds = \int_0^t \mathbbm{1}_{\{X(s) =0\}} dr^{-1}(s)
 \end{equation}
 By \eqref{eq:r-inv}, \eqref{eq:occu=0}, and \eqref{eq:Steiltjes},
 \begin{equation*}
   \int_0^t \mathbbm{1}_{\{X(s) = 0\}} ds - \int_0^t \mathbbm{1}_{\{X(s) =0\}} \alpha L^X(ds,0) = 0
 \end{equation*}
 and because $s \mapsto L^X(s,0)$ only grows when $X(s)=0$,
 \[\int_0^t \mathbbm{1}_{\{X(s) =0\}}  L^X(ds,0) = \alpha L^X(t,0).\]
 \end{proof}

 This occupation time formula gives us a simpler representation for the time change $r^{-1}$. In \eqref{eq:r-inv}, we showed that a delayed process is a time changed version of a process with no delay, using the time change $r^{-1}(t) = t - \alpha L^X(t,0)$. According to Theorem \ref{thm:occupation-time}, this time change can also be written as
 \[r^{-1}(t) = t - \int_0^t \mathbbm{1}_{\{X(s)=0\}}ds = \int_0^t \mathbbm{1}_{\{X(s) \not =0\}}ds.\]
 This time change is, therefore, absolutely continuous with respect to the Lebesgue measure which allows us to write $X(t)$ in terms of a simpler SDE. Recall that the local time of a process without delay (for example, the Brownian local time) is not absolutely continuous with respect to Lebesgue measure.

 \begin{thm} \label{T:SDE-representation}
   Let $X(t)$ be the diffusion associated with the infinitesimal operator \eqref{eq:infites-gen}-\eqref{eq:boundary-conds}. Then $X(t)$ is a weak solution to the SDE
   \begin{equation} \label{eq:simplest-SDE}
     dX(t) = b(X(t))\mathbbm{1}_{\{X(t) \not =0\}}dt + \sigma(X(t)) \mathbbm{1}_{\{X(t) \not = 0\}} d\widetilde{W}(t) + (p_+-p_-)L^X(dt,0).
   \end{equation}
   for some Brownian motion $\widetilde{W}(t)$.
 \end{thm}
 \begin{proof}
   Let $Y(t)$ be the solution to the undelayed process
    \[dY(t) = b(Y(t))dt + \sigma(Y(t))dW(t) + (p_+-p-)L^Y(t,0)\]
    and let $r(t) = t + \alpha L^Y(t,0)$. Then by Theorem \ref{thm:occupation-time},
   \[r^{-1}(t) = t - \alpha L^X(t,0) = \int_0^t \mathbbm{1}_{\{X(s) \not =0\}}ds.\]
   From the arguments in Section \ref{S:One-d-delay-reflection}, $X(t):=Y(r^{-1}(t))$ is a solution to the martingale problem associated with \eqref{eq:infites-gen}-\eqref{eq:boundary-conds}. By Theorem \ref{thm:occupation-time},
   \[dX(t) = b(X(t))\mathbbm{1}_{\{X(t) \not=0\}}dt + \sigma(X(t))dV(t) +(p_+-p_-)L^X(dt,0). \]
   In the above equation $V(t) = W\left(\int_0^t \mathbbm{1}_{\{X(s) \not =0\}} ds \right)$.
   To finish the proof, we define the Brownian motion
   \[\widetilde{W}(t) =  W\left(\int_0^t \mathbbm{1}_{\{X(s) \not =0\}}ds\right) + W_2\left(\int_0^t \mathbbm{1}_{\{X(s) = 0\}}ds \right)\]
   where $W_2$ is a Brownian motion that is independent of $W$. In this way,
   \[V(t) = \int_0^t \mathbbm{1}_{\{X(s) \not =0\}} d\widetilde{W}(s)\]
   and the conclusion follows.
 \end{proof}
 \begin{rem}
   Observe that the solutions to SDE \eqref{eq:simplest-SDE} are not unique. In particular, this equation is satisfied by the delayed equation with any delay parameter $\alpha\geq 0$. This SDE is trivially satisfied by the classic undelayed process ($\alpha = 0$) because the process spends almost no time at $0$. This equation is also satisfied by the absorbing process $Y(t\wedge\tau)$ where $\tau = \inf\{t>0:Y(t)=0\}$. Despite the lack of uniqueness, \eqref{eq:simplest-SDE} demonstrates that delayed Markov processes are semimartingales and gives an explicit form for their Doob decomposition. All of the classical semimartingale theory follows including Meyer-Tanaka formula and Girsanov formula.
 \end{rem}
 If we combine Theorems \ref{thm:occupation-time} and \ref{T:SDE-representation}, then we do have uniqueness as the following theorem shows.
 \begin{thm} \label{thm:SDE-local-time}
   The SDE and local time pair
   \begin{align}
      dX(t) &= b(X(t))\mathbbm{1}_{\{X(t) \not=0\}} dt+ \sigma(X(t)) \mathbbm{1}_{\{X(t) \not =0\}} dW(s)+(p_{+}-p_{-})L^X(dt,0), \nonumber\\
      X(0)&=x\nonumber\\
      \alpha L^X(t,0) &= \int_0^t \mathbbm{1}_{\{X(s)=0\}}ds\label{eq:SDE-local-time-pair}
    \end{align}
   has a solution that is unique in law.
 \end{thm}
 Before proving Theorem \ref{thm:SDE-local-time}, we prove the following lemma.
 \begin{lem} \label{lem:explicit-martingale-prob}
   Let $X(t)$ solve \eqref{eq:SDE-local-time-pair} and let $\mathscr{L}$ be given by \eqref{eq:infites-gen}. Assume that $f$ is a function that is twice continuously differentiable on $\mathbb{R}\setminus \{0\}$, its first and second derivatives have left and right limits at $0$, and $\mathscr{L}f$ is continuous at $0$. Then
   \begin{align} \label{eq:explicit-martingale-formula}
    f(X(t)) - f(x) = &\int_0^t \mathscr{L}f(X(s))ds + \int_0^t \sigma(X(s))\mathbbm{1}_{\{X(s) \not = 0\}}dW(s)  \nonumber\\
    &+(p_+f'(0+)-p_-f'(0-) - \alpha \mathscr{L}f(0))L^X(t,0).
   \end{align}
 \end{lem}
 \begin{proof}
   We will use the Meyer-Tanaka formula (Theorem \ref{T:Meyer-Tanaka}). Note that because $f'(x)$ has a jump discontinuity at $x=0$, and the second derivative exists everywhere except for $x=0$, then the second derivative measure of $f$ is $\mu(dx) = f''(x) dx + (f'(0+) - f'(0-))\delta_0(dx)$, where $\delta_0$ is the delta Dirac measure at $0$. Therefore, by Corollary \ref{Cor:Meyer-Tanaka},
    \begin{align*}
    \frac{1}{2} \int_{-\infty}^\infty \mu(dy) L^X(t,y) &= \frac{1}{2} \int_0^t f''(X(s))\sigma^2(X(s))\mathbbm{1}_{\{X(s)\not=0\}} ds \nonumber\\
    &\qquad+ \frac{1}{2}(f'(0+) - f'(0-))L^X(t,0).
    \end{align*}

    Then by Theorem \ref{T:Meyer-Tanaka},
   \begin{align*}
    &f(X(t)) -f(x) = \int_0^t f'(X(s))b(X(s))\mathbbm{1}_{\{X(s) \not = 0\}} ds  \\
    &+\frac{1}{2}(p_+ - p_-)(f'(0+)+f'(0-))L^X(t,0)+\frac{1}{2} \int_0^t f''(X(s))\sigma^2(X(s))\mathbbm{1}_{\{X(s) \not =0\}}ds \\
    &+ \frac{1}{2}(f'(0+)-f'(0-))L^X(t,0)+ \int_0^t f'(X(s)) \sigma(X(s))\mathbbm{1}_{\{X(s) \not =0\}}dW(s)\\
    &= \int_0^t \mathscr{L}f(X(s)) \mathbbm{1}_{\{X(s) \not = 0\}} ds + \int_0^t f'(X(s))\sigma(X(s))\mathbbm{1}_{\{X(s) \not = 0\}} dW(s)\\
    & + \left(p_+f'(0+) - p_- f'(0-) \right)L^X(t,0).
   \end{align*}
   Finally, we use the fact that $\int_0^t \mathbbm{1}_{\{X(s)\not=0\}}ds = t - \alpha L^X(t,0)$ and the continuity of $\mathscr{L}f$ to write
   \begin{align}
     f(X(t))- f(x) &= \int_0^t \mathscr{L}f(X(s)) ds + \int_0^t \sigma(X(s))\mathbbm{1}_{\{X(s) \not=0\}}dW(s) \nonumber\\
     &\quad+ \left( p_+ f'(0+)-p_- f'(0-) -\alpha \mathscr{L}f(0) \right) L^X(t,0).\nonumber
   \end{align}
 \end{proof}
 \begin{proof}[Proof of Theorem \ref{thm:SDE-local-time}]
   We proved the existence of solutions to \eqref{eq:SDE-local-time-pair} in Theorems \ref{thm:occupation-time} and \ref{T:SDE-representation}. We prove uniqueness in law by first showing that any solution to \eqref{eq:SDE-local-time-pair} is a Markov process  whose infinitesimal generator is \eqref{eq:infites-gen} with boundary conditions \eqref{eq:boundary-conds}. Uniqueness follows from the Hille-Yosida Theorem, see for example Theorem 1.4.3 in \cite{EithierKurtz}.
   Let $f$ satisfy the boundary conditions \eqref{eq:boundary-conds}. That is, $f$ is continuous and $\mathscr{L}f$ is continuous. The first and second derivatives of $f$ exist and are continuous everywhere except for maybe at $0$, but the first derivative has limits from the right and left at $0$.
   By Lemma \ref{lem:explicit-martingale-prob},
   \begin{align*}
     \E(f(X(t))) &- f(x) = \E\int_0^t \mathscr{L}f(X(s))ds \\
     &(p_+f'(0+)-p_-f'(0-)-\alpha \mathscr{L}f(0))\E L^X(t,0).
   \end{align*}
   We assumed that $f$ satisfied the boundary conditions \eqref{eq:boundary-conds} and we can conclude that $\mathscr{L}$ is indeed the infinitesimal generator of the process $X$.
   The uniqueness of the martingale problem associated to the operator $\mathscr{L}$ implies the uniqueness in law of the solution to \eqref{eq:SDE-local-time-pair}.
 \end{proof}

 We conclude this section by stating some consequences of Theorem \ref{thm:SDE-local-time}.

 \begin{cor}[Girsanov's Theorem]\label{C:Girsanov}
   Let $\phi:\mathbb{R} \to \mathbb{R}$.  Assume that $X(t)$ is a solution to \eqref{eq:SDE-local-time-pair} on a probability space $(\Omega, \mathcal{F},\Pro)$. Let $\phi$ be such that $\mathbb{Q}$ defined by
   \[\frac{d\mathbb{Q}}{d\mathbb{P}} = e^{\int_0^T \phi(X(s))dW(s) - \frac{1}{2} \int_0^T |\phi(X(s))|^2 ds}.\]
   is probability measure on $(\Omega,\mathcal{F})$. Under this new probability measure,
   \[\widetilde{W}(t) = W(t) - \int_0^t \phi(X(s))ds\]
   is a Wiener process.
   Furthermore, $X(t)$ solves the equation
   \begin{equation*}
     \begin{cases}
       dX(t) = (b(X(t)) + \sigma(X(t))\phi(X(t)))\mathbbm{1}_{\{X(s)\not=0\}}ds + \sigma(X(s))\mathbbm{1}_{\{X(s)\not=0\}}d\widetilde{W}(s)\\
       \quad+(p_+-p_-)L^X(dt,0),\\
       X(0)=x,\\
       \alpha L^X(t,0) = \int_0^t \mathbbm{1}_{\{X(s)=0\}} ds.
     \end{cases}
   \end{equation*}
   That is, under the measure $\mathbb{Q}$, $X$ is a solution to the martingale problem associated to the generator
   \[\widetilde{\mathscr{L}}f(x) = \mathscr{L}f(x) + \sigma(x)\phi(x)f'(x)\]
   with boundary conditions \eqref{eq:boundary-conds}.
 \end{cor}

If coefficients $b$ and $\sigma$ are continuous, then alternative, perhaps more compact, forms are possible as seen in Corollary \ref{cor:cont-coeff}.
 \begin{cor}[SDE representation if $b$ and $\sigma$ are continuous] \label{cor:cont-coeff}
   If $b$ and $\sigma$ are continuous and $X(t)$ solves \eqref{eq:SDE-local-time-pair}, then
   $X(t)$ is the unique weak solution of
   \begin{equation}
     dX(t) = b(X(t))dt- b(0)\alpha L^X(dt,0) + \sigma(X(t))dV^X(t) + (p_+-p_-)L^X(dt,0)\label{Eq:SDE-continuousCoefficients}
   \end{equation}
   where $V^X(t)=\int_0^t\mathbbm{1}_{\{X(s)\not=0\}}dW(s)$ is a martingale with quadratic variation $\langle V^X \rangle_t = t-\alpha L^X(t,0)$.
   For any bounded twice continuously differentiable function $f$,
   \begin{align*}
     f(X(t)) - f(x) = &\int_0^t \mathscr{L}f(X(s))ds + \int_0^t f'(X(s))\sigma(X(s))dV^X(s) \nonumber \\
      &(p_+f'(0)-p_-f'(0) - \alpha \mathscr{L}f(0))L^X(t,0).
   \end{align*}
 \end{cor}

\begin{rem}
At this point we remark that (\ref{eq:SDE-local-time-pair}) or (\ref{Eq:SDE-continuousCoefficients}) have unique weak solutions. One does not expect to have strong solutions or pathwise uniqueness, see the recent works of \cite{Bass2014,EngelbertPeskir2014}, as well as the older works of \cite{KaratzasShirayevShkolnikov2011} and of \cite{Warren1,Warren2,Watanabe} on sticky Brownian motion.
\end{rem}
\section{Relation to Feller and Volkonskii results}\label{S:DvDu}
In \cite{Feller}, Feller showed that under minimal regularity conditions, all one-dimension diffusion generators can be represented in the form $D_vD_u$ where $v$ and $u$ are strictly increasing functions. $u$ is continuous and $v$ is allowed to have jump discontinuities. In addition, $D_{u}$, $D_{v}$ are differentiation
operators with respect to $u(x)$ and $v(x)$ respectively, which are
defined as follows:

$D_{u}f(x)$ exists if $D_{u}^{+}f(x)=D_{u}^{-}f(x)$, where the left
derivative of $f$ with respect to $u$ is defined as follows:
\begin{displaymath}
D_{u}^{-}f(x)=\lim_{h\downarrow 0}\frac{f(x-h)-f(x)}{u(x-h)-u(x)}
\hspace{0.2cm} \textrm{ provided the limit exists.}
\end{displaymath}
The right derivative $D_{u}^{+}f(x)$ is defined similarly. If $v$ is
discontinuous at $y$ then
\begin{displaymath}
 D_{v}f(y)=\lim_{h\downarrow 0}\frac{f(y+h)-f(y-h)}{v(y+h)-v(y-h)}.
\end{displaymath}

The $D_vD_u$ operator along with its domain of definition uniquely characterize the distribution of a one-dimensional Markov process.

In this section we firstly see how to use Theorem \ref{thm:SDE-local-time} to give an SDE representation for a large class of $D_{v}D_{u}$ processes. Because the two functions $u,v$ are strictly increasing, they are both differentiable except at a finite or countable number of points.
Let $E=\{x_i\}_{i\geq 1}$ be the set of points of non-differentiability for $u$ and $v$. For each $x_i \in E$, let
\begin{align*}
  &p_+^i = \frac{u'(x_i-)}{u'(x_i+)+u'(x_i-)},\qquad p_-^i = \frac{u'(x_i+)}{u'(x_i+)+u'(x_i-)},\label{Eq:ParametersBoundaryCondition}\\
  &\alpha_i = \frac{(v(x_i+) - v(x_i-))u'(x_i+)u'(x_i-)}{u'(x_i+)+u'(x_i-)}.\nonumber
\end{align*}

The domain of definition of $D_vD_u$ is
\begin{equation*} \label{eq:domain-of-DvDu}
  \textnormal{Dom}(D_vD_u) = \left\{\begin{array}{l}
  \ds{f \in C^2(\mathbb{R} \setminus E) \cap C(\mathbb{R}): D_vD_u f \in C(\mathbb{R}),}\\
  \ds{p_+^if'(x_i) - p_-^if'(x_i-) = \alpha_i D_vD_uf(x_i), \ x_i \in E}
  \end{array} \right\}.
\end{equation*}

We make the following assumption.
\begin{ass}\label{A:RegularityOn_u_and_v}
Let us assume that the function $u(x)\in\mathcal{C}^{2}\left(\mathbb{R} \setminus E\right)$ and the function  $v(x)\in\mathcal{C}^{1}\left(\mathbb{R} \setminus E\right)$.
\end{ass}

Under Assumption \ref{A:RegularityOn_u_and_v} we define the drift and diffusion functions $b$ and $\sigma$ as maps from $\mathbb{R} \setminus E \to \mathbb{R}$ by
\begin{align}
  &b(x) = -\frac{u''(x)}{(u'(x))^2 v'(x)},\qquad \sigma(x) = \sqrt{\frac{2}{u'(x)v'(x)}}\label{Eq:Drift-Diffusion}
\end{align}

Let us assume now that $Y$ solves the SDE
\begin{equation*}
  dY(t) = b(Y(t))dt + \sigma(Y(t))dt + \sum_{x_i\in E} (p_+^i - p_-^i) L^Y(dt,x_i).
\end{equation*}
Define the time change
\begin{equation*}
  r(s) = s + \sum_{x_i\in E} \alpha_i L^Y(s,x_i)
\end{equation*}
and let
\begin{equation*}
  X(t) = Y(r^{-1}(t)).
\end{equation*}

Then $X(t)$ is a diffusion process with infinitesimal generator $D_vD_u$ and based on Theorem \ref{thm:SDE-local-time},
 $X$ is a weak solution to the SDE, local time pair
\begin{equation*}\label{eq:SDE-local-time-many-points}
    \begin{cases}
     &dX(t) = b(X(t))\mathbbm{1}_{\{X(t) \not \in E\}}dt + \sigma(X(t))\mathbbm{1}_{\{X(t) \not \in E\}}dW(t)\\
     &+ \sum_{x_i\in E}\left(p_+^i - p_-^i\right)L^X(dt,x_i)\\
     &X(t) = x,\\
     &\alpha_i L^X(t,x_i) = \int_0^t \mathbbm{1}_{\{X(s) = x_i\}}ds, \ x_i \in E.
  \end{cases}
\end{equation*}

We conclude this section by connecting our results to those of Volkonskii \cite{Volkonskii1,Volkonskii2}. It is proven in \cite{Volkonskii1} that any one dimensional, homogeneous, strong Markov process with infinitesimal generator $D_v D_u$ can be represented as
\begin{equation*}
  X(t) = u^{-1}(W(\tau^{-1}(t))).
\end{equation*}
where $W(t)$ is a one-dimensional Brownian motion. The time change $\tau(t)$ is defined as the limit of
\begin{equation*}
  \tau_n(t) = \int_0^t \frac{dv_n}{du}(u^{-1}(W(s))) ds
\end{equation*}
where $v_n$ is a sequence of differentiable functions with respect to $u$ and $v_n \rightharpoonup v$ weakly. In order to connect this result to Theorem \ref{T:GeneralTimeChangedMP_1}, we rewrite $\tau(t)$ as a composition of two time changes. Recall that $v(t)$ is an increasing function and let $V_d$ be the set of points of discontinuity of $v$ (these are all the possible points where the process can have delay). Define $\tilde{v}$ to be the continuous function
\begin{equation*}
  \tilde{v}(x) =
    \begin{cases}
      \tilde{v}(0) = v(0) \\
      \tilde{v}(x) = v(x) + \sum_{\{y \in V_d: x < y \leq 0\}} (v(y+)-v(y-)) & \text{ if } x<0\\
      \tilde{v}(x) = v(x) - \sum_{\{y \in V_d: 0 < y \leq x\}} (v(y+)-v(y-)) & \text{ if } x<0\\
    \end{cases}
\end{equation*}
That is $\tilde{v}$ is $v$ with all of the jump discontinuities removed. Define
\begin{equation*}
  \tilde{\tau}(t) = \int_0^t \frac{d\tilde{v}}{du}(u^{-1}(W(s))) ds
\end{equation*}
and let
\begin{equation}
  Y(t) = u^{-1}(W(\tilde{\tau})).\label{Eq:VolkonskiiRep1}
\end{equation}

By Feller's result, we know that there are choices of $\tilde{v}(x), u(x)$ such that the process $Y$ from (\ref{Eq:VolkonskiiRep1}) coincides in distribution with that of (\ref{eq:SDE-Y}). Define
\begin{equation*}
  r(t) = t +  \sum_{z \in V_d} (v(z+) - v(z-))/([u'(z+)]^{-1}+[u'(z-)]^{-1}) L^Y(t,z).
\end{equation*}

The delayed system, $X$ is then given by
\[X(t) = Y(r^{-1}(t)).\]
In this case,
\[\tau(t) = r(\tilde{\tau}(t))\]
and it follows that
\[X(t) = Y(r^{-1}(t)) = u^{-1}(W(\tau^{-1}(t))),\]
as desired.

\section{An example: Wiener process with reflection in narrow tubes}\label{S:Example}
In this section we present a concrete physical example that gives rise to a $D_{v}D_{u}$ process with potential delay at a point. Then we use Theorem \ref{thm:SDE-local-time} to represent the stochastic process as a solution to an SDE.

In \cite{SpilEJP2009}, a Wiener process with instantaneous reflection in narrow tubes of
width $\varepsilon\ll 1$ around axis $x$ is considered. The tube is assumed to be (asymptotically) non-smooth in that if $V^{\varepsilon}(x)$ denotes the volume of the
cross-section of the tube, then
$\frac{1}{\varepsilon}V^{\varepsilon}(x)$ converges in an appropriate
sense to a non-smooth function as
$\varepsilon\downarrow 0$. Then, as it is characterized in \cite{SpilEJP2009}, depending on the behavior of $\frac{1}{\varepsilon}V^{\varepsilon}(x)$ as $\varepsilon\downarrow 0$, one gets in the limit as $\varepsilon\downarrow 0$ a $D_{v}D_{u}$ process. Let us be more specific now.

For each $x\in \mathbb{R}$ and $0<\varepsilon<<1$, let
$D^{\varepsilon}_{x}$ be a bounded interval in $\mathbb{R}$  that
contains $0$.  Consider the state space
$D^{\varepsilon}=\{(x,y):x\in \mathbb{R},y\in D^{\varepsilon}_{x}\}\subset
\mathbb{R}^{2}$. Assume that the boundary $\partial D^{\varepsilon}$ of
$D^{\varepsilon}$ is smooth enough and denote by
$\gamma^{\varepsilon}(x,y)$ the inward unit normal to $\partial
D^{\varepsilon}$. Assume that $\gamma^{\varepsilon}(x,y)$ is not parallel
to the $x$-axis.

Denote by $V^{\varepsilon}(x)$ the length of the cross-section
$D_{x}^{\varepsilon}$ of the stripe and assume that  $V^{\varepsilon}(x)\downarrow
0$ as $\varepsilon \downarrow 0$. In addition, we assume that
$\frac{1}{\varepsilon}V^{\varepsilon}(x)$ converges in an appropriate
sense to a non-smooth function, $V(x)$, as $\varepsilon\downarrow 0$.
The limiting function can be composed for example by smooth
functions, step functions and also the Dirac delta distribution.

Consider the Wiener process $(X^{\varepsilon}(t),Y^{\varepsilon}(t))$ in $D^{\varepsilon}$ with instantaneous normal reflection on the boundary of $D^{\varepsilon}$. Its trajectories can be described by the stochastic differential equations:
\begin{eqnarray*}
X^{\varepsilon}(t)&=& x+ W^{1}(t)+\int_{0}^{t}\gamma_{1}^{\varepsilon}(X^{\varepsilon}(s),Y^{\varepsilon}(s))dL^{\varepsilon}(s)\nonumber\\
Y^{\varepsilon}(t)&=& y + W^{2}(t)+\int_{0}^{t}\gamma_{2}^{\varepsilon}(X^{\varepsilon}(s),Y^{\varepsilon}(s))dL^{\varepsilon}(s).\label{StochasticProcessWithReflection1}
\end{eqnarray*}
Here  $W^{1}(t)$ and $W^{2}(t)$ are independent Wiener processes in $\mathbb{R}$  and $(x,y)$ is a point inside $D^{\varepsilon}$; $\gamma_{1}^{\varepsilon}$ and $\gamma_{2}^{\varepsilon}$ are both projections of  the unit inward normal vector to $\partial D^{\varepsilon}$ on the axis $x$ and $y$ respectively. Furthermore, $L^{\varepsilon}(t)$ is the local time for the process $(X^{\varepsilon}(t),Y^{\varepsilon}(t))$ on $\partial D^{\varepsilon}$, i.e. it is a continuous, non-decreasing process that increases only when $(X^{\varepsilon}(t),Y^{\varepsilon}(t)) \in \partial D^{\varepsilon}$ such that the Lebesgue measure $\Lambda\{t>0:(X^{\varepsilon}(t),Y^{\varepsilon}(t)) \in \partial D^{\varepsilon}\}=0$ (eg. see \cite{KS}).

As it is shown in \cite{SpilEJP2009}, if $\frac{1}{\varepsilon}V^{\varepsilon}(x)=V(x)$, where $V(x)$ is a smooth
function then $X^{\varepsilon}(t)$ converges to a standard diffusion process $X(t)$, as $\varepsilon\downarrow 0$. In particular,  for any $T>0$
\begin{equation*}
\sup_{0\leq t \leq T}E_{x}|X^{\varepsilon}(t)-X(t)|^{2}\rightarrow 0 \hspace{0.2cm} \textrm{as} \hspace{0.2cm} \varepsilon \rightarrow 0,
\end{equation*}
where  $X(t)$ is the solution of the stochastic differential equation
\begin{equation*}
X(t)= x + W^{1}(t) + \int_{0}^{t}\frac{1}{2}\frac{V_{x}(X(s))}{V(X(s))}ds \label{LimitingStochasticProcess1}
\end{equation*}
and $V_{x}(x)=\frac{d V(x)}{dx}$.

Let us assume now that
$\frac{1}{\varepsilon}V^{\varepsilon}(x)$ converges to a non-smooth
function as $\varepsilon \downarrow 0$. Owing to the non smoothness of the limiting function, one cannot hope to obtain a limit in mean square sense to a standard diffusion process as before. In particular, as it is proven in \cite{SpilEJP2009}, the non smoothness of the limiting function leads to the effect that the limiting diffusion may have points where the scale function is not differentiable (skew diffusion) and also points with positive speed measure (points with delay).

Introduce the functions
\begin{equation*}
u^{\varepsilon}(x)=\int_{0}^{x}2\frac{\varepsilon}{V^{\varepsilon}(y)}dy \hspace{0.3cm}\textrm{ and }\hspace{0.3cm}
v^{\varepsilon}(x)=\int_{0}^{x}\frac{V^{\varepsilon}(y)}{\varepsilon}dy.\label{uANDvForEscortProcessIntro}
\end{equation*}

and assume that the functions
\begin{eqnarray*}
u(x)&=&\lim_{\varepsilon \downarrow 0 }u^{\varepsilon}(x) \textrm{, } x\in\mathbb{R}\nonumber \\
v(x)&=&\lim_{\varepsilon \downarrow 0 }v^{\varepsilon}(x) \textrm{, } x\in\mathbb{R}\setminus \lbrace 0\rbrace ,\label{uANDvForLimitingProcessIntro}
\end{eqnarray*}
are well defined and  the limiting function $u(x)$ is continuous and strictly increasing whereas the limiting function $v(x)$ is right continuous and strictly increasing. In general, the function $u(x)$ can have countably many
points where it is not differentiable and the function $v(x)$ can have countably
many points where it is not continuous or not differentiable. However,
here we assume for brevity that the only  non smoothness point is $x=0$.

Then, we have the following theorem.
\begin{thm}[Theorem 1.2 in \cite{SpilEJP2009}]
 Let $X$ be the solution to the martingale problem for
\begin{equation}
A=\lbrace(f,\mathscr{L}f):f\in \mathcal{D} (A)\rbrace \label{LimitingProcess1}
\end{equation}
with
\begin{equation*}
\mathscr{L}f(x)=D_{v}D_{u}f(x)\label{LimitingOperator0}
\end{equation*}
and
\begin{align*}
\textnormal{Dom}(A)&=\lbrace f: f\in \mathcal{C}_{c}(\mathbb{R})\textrm{, with }f_{x},f_{xx} \in \mathcal{C}(\mathbb{R}\setminus\lbrace 0\rbrace),\nonumber\\
&  [u'(0+)]^{-1}f_{x}(0+)-[u'(0-)]^{-1} f_{x}(0-)=[v(0+)-v(0-)] \mathscr{L}f(0) \nonumber\\
&   \textrm{ and } \mathscr{L}f(0)=\lim_{x\rightarrow
0^{+}}\mathscr{L}f(x)=\lim_{x\rightarrow 0^{-}}\mathscr{L}f(x) \rbrace. \label{LimitingOperator0Condition1}
\end{align*}

Then we have
\begin{equation*}
X^{\varepsilon}\longrightarrow X\textrm{ weakly in } \mathcal{C}_{0T}, \textrm{ for any } T<\infty, \textrm{ as } \varepsilon\downarrow 0,\label{Claim1}
\end{equation*}
where $\mathcal{C}_{0T}$ is the space of continuous functions in $[0,T]$.
\begin{flushright}
$\square$
\end{flushright}\label{Theorem11}
\end{thm}

As proved in Feller \cite{Feller} the martingale problem for $A$,  (\ref{LimitingProcess1}), has a unique
solution $X$. It is an asymmetric
Markov process with delay at the point of discontinuity $0$. In
particular, the asymmetry is due to the possibility of having $u'(0+)\neq u'(0-)$  whereas the delay  is because of the possibility of having $v(0+)\neq v(0-)$. To make the discussion more concrete, let us assume that $V^{\varepsilon}(x)$ can be decomposed as follows
\begin{equation}
V^{\varepsilon}(x)=V^{\varepsilon}_{1}(x)+V^{\varepsilon}_{2}(x)+V^{\varepsilon}_{3}(x),\label{DefinitionOfCrossSections1special}
\end{equation}
where the functions $V^{\varepsilon}_{i}(x)$, for $i=1,2,3$, satisfy the following conditions:
\begin{enumerate}
\item{$V^{\varepsilon}_{1}(x)=\varepsilon V_{1}(x)$, where $V_{1}(x)$ is any smooth, strictly positive function,}
\item{$V^{\varepsilon}_{2}(x)=\varepsilon V_{2}(\frac{x}{\delta})$, such that $V_{2}(\frac{x}{\delta})\rightarrow\beta \chi_{\{x>0\}}$ with $\beta\geq 0$, uniformly for every connected subset of $\mathbb{R}$ that is away from an arbitrary small neighborhood of $0$ and weakly within a neighborhood of $0$, as  $\varepsilon\downarrow 0$.}
\item{$V^{\varepsilon}_{3}(x)=\frac{\varepsilon}{\delta}V_{3}(\frac{x}{\delta})$, such that $\frac{1}{\delta}V_{3}(\frac{x}{\delta})\rightarrow\mu \delta_{0}(x)$, in the weak sense as  $\varepsilon\downarrow 0$. Here $\mu$ is a nonnegative constant and $\delta_{0}(x)$ is the  Dirac delta
distribution at $0$.}
\end{enumerate}

Let us define $\gamma=V_{1}(0)$ and notice that $\mu=\int_{-\infty}^{\infty}V_{3}(x)dx$.
Then, combining Theorems  \ref{thm:SDE-local-time} and \ref{Theorem11}, we get the following Corollary for the situation just described.

\begin{cor}
Consider, the set-up of Theorem \ref{Theorem11}, assume Assumption \ref{A:RegularityOn_u_and_v} and let $b(x)$, $\sigma(x)$ be given by relations (\ref{Eq:Drift-Diffusion}) via the limiting $v(x)$ and $u(x)$ of (\ref{uANDvForLimitingProcessIntro}) with $V^{\varepsilon}(x)$ as in (\ref{DefinitionOfCrossSections1special}). Then, the limiting process $X(t)$ can be equivalently characterized as the weak solution to the SDE
\begin{align}
  dX(t) &= b(X(t))\mathbbm{1}_{\{X(t) \not=0\}}dt + \mathbbm{1}_{\{X(t) \not=0\}}dW(t) +
   \frac{\beta}{2\gamma + \beta} L^X(dt,0)
\end{align}
where
\[
b(x)=\frac{1}{2}\frac{d}{dx}[\ln (V_{1}(x) )]\chi_{\{x\leq 0\}}+\frac{1}{2}\frac{d}{dx}[\ln (V_{1}(x)+\beta )]\chi_{\{x> 0\}}.
\]
Moreover, $p_{+}=\frac{\gamma+\beta}{2\gamma+\beta}, p_{-}=\frac{\gamma}{2\gamma+\beta}$, $\alpha=\frac{2\mu}{2\gamma+\beta}$ and for the occupation time we have the formula
\begin{equation}
    \int_0^t \mathbbm{1}_{\{X(s) = 0\}} ds = \frac{2\mu}{2\gamma+\beta} L^X(t,0).
  \end{equation}
\end{cor}

\section{Acknowledgements}
We would like to thank Professor Ioannis Karatzas for making known to us, upon completion of this work, of the recent results of \cite{Bass2014, EngelbertPeskir2014}.

\begin{appendices}
\section{Symmetric local time} \label{S:local time}
In this section we review for completeness the definition and some of the properties of symmetric local time. This material is classical, see for example \cite{RevuzYor}.
\begin{defn}
For any semimartingale, $X(t)$, the symmetric local time of  $X$ at $y$ is defined as
\begin{equation*} \label{def:local_time}
  L^X(t,y) = |X(t) - y| - |x-y| - \int_0^t \text{sign}(X(s) - y) dX(s).
\end{equation*}
\end{defn}
Another equivalent definition of local time, which helps explain why we use the word ``symmetric,'' defines the process in terms of scaled symmetric occupation times
\[L^X(t,y) = \lim_{\delta \downarrow 0} \frac{1}{2\delta}\int_0^t \mathbbm{1}_{[-\delta,\delta]}(X(s))d\langle X \rangle_s.\]
The right local time can be defined as
\[L^X(t,y+) = \lim_{\delta \downarrow 0} \frac{1}{\delta}\int_0^t \mathbbm{1}_{[0,\delta](X(s))}d\langle X \rangle_s\]
and the left local time can be defined as
\[L^X(t,y-) = \lim_{\delta \downarrow 0} \frac{1}{\delta}\int_0^t \mathbbm{1}_{[-\delta,0](X(s))}d\langle X \rangle_s.\]

\begin{thm}[Properties of symmetric local time]
  For any $y \in \mathcal{R}$,
  \begin{enumerate}
    \item[(i)] $t \mapsto L^X(t,y)$ is an increasing process with probability $1$.
    \item[(ii)] $t \mapsto L^X(t,y)$ is constant on any interval on which $X(t) \not = y$.
  \end{enumerate}
\end{thm}

Next we recall the Meyer-Tanaka formula, which generalized the It\^{o} formula (see \cite{Lejay}).
\begin{thm}[Meyer-Tanaka formula] \label{T:Meyer-Tanaka}
  For any function $f$ with left and right derivatives $f_l'$ and $f_r'$ and second derivative measure $\mu$,
  \begin{equation*}
    f(X(t)) -f(X(0)) = \int_0^t \frac{1}{2}(f_l'(X(s)) + f_r'(X(s)))dX(s) + \frac{1}{2}\int_{-\infty}^\infty \mu(dy) L^X(t,y).
  \end{equation*}
\end{thm}
\begin{cor} \label{Cor:Meyer-Tanaka}
  For any function $f$,
  \begin{equation*}
    \int_0^t f(X(s))d\langle X \rangle_s = \int_{-\infty}^\infty f(y) L^X(t,y)dy.
  \end{equation*}
\end{cor}

\section{Distributional properties of delayed Brownian motion}\label{S:BrownianMotionDelay_Properties}
In this section we study the properties and the distribution of a delayed Brownian motion, also known as sticky Brownian motion. For results in the case $\alpha=1$ see \cite{Amir1991}. A delayed Brownian motion is a diffusion with infinitesimal generator
\[\mathscr{L}f(x) = \frac{1}{2}f''(x)\]
and boundary conditions
\[\frac{1}{2}f'(0+)-\frac{1}{2}f'(0-) = \alpha \mathscr{L}f(0),\]
for $\alpha>0$. This is the simplest example of a Markov process with spatial delay. We showed in Section \ref{S:One-d-delay-reflection} that such a process can be expressed as the time-changed Brownian motion
\begin{equation*} \label{eq:delayed-BM}
  X(t) = W(r^{-1}(t))
\end{equation*}
where $r(t) = t + \alpha L^W(t,0)$ and $r^{-1} =t - \alpha L^X(t,0)$ is its functional inverse.
We will now characterize the distribution of $X(t)$ for any given $t$.

First, we recall that for fixed time $t$, the Brownian local time $L^W(t,0)$ has the same distribution as  $M(t) :=-\inf_{s\leq t}W(t)\wedge 0$. Furthermore, the distribution of the running maximum is known (see for example \cite{KS,RevuzYor}). These results are summarized in the following lemma.
\begin{lem}[Distribution of Brownian local time] \label{lem:BM-local-dist}
  For any $t>0$, $y >0$,
  \begin{equation*}
    \Pro\left( L^W(t,0) > y \right)  = \frac{\sqrt{2}}{\sqrt{\pi t}}\int_y^\infty e^{-\frac{z^2}{2t}}dz.
  \end{equation*}
\end{lem}

Then we can characterize  the distribution of $L^X(t,0)$. Recall that $\alpha L^X(t,0)\leq t$.
\begin{lem} \label{lem:local-time-dist}
  For any $t>0$, $0\leq y<t$,
  \begin{equation*}
    \Pro\left(\alpha L^X(t,0) > y \right) = \Pro(\alpha L^W(t- y,0)>y) = \frac{\sqrt{2}}{\sqrt{\pi(t- y)}} \int_{y/\alpha}^\infty e^{-\frac{z^2}{2(t- y)}}dz.
  \end{equation*}
\end{lem}
\begin{proof}
  This proof follows from the characterizations of $r$ and $r^{-1}$ \eqref{eq:r-def} and \eqref{eq:r-inv}.
  First observe that
  \[\Pro(\alpha L^X(t,0)>y) = \Pro(t-\alpha L^X(t,0)< t- y) = \Pro(r^{-1}(t)< t -  y).\]
  Then because $r$ is an increasing continuous function, the above expression equals
  \[=\Pro(t< r(t- y)) = \Pro(t < t- y + \alpha L^W(t- y,0)).\]
  This simplifies to
  \[=\Pro(L^W(t- y,0)>y/\alpha).\]
  The result is then a consequence of Lemma \ref{lem:BM-local-dist}.
\end{proof}

We showed in Theorem \ref{thm:occupation-time} that the occupation time of a delayed Brownian motion can be characterized by
\[\int_0^t \mathbbm{1}_{\{X(s) = 0\}}ds = \alpha L^X(t,0).\]
In this way, the previous lemma also characterizes the distribution of the occupation time.

\begin{lem}[Expected Occupation Time]
  The expected occupation time is
  \begin{equation} \label{eq:expected-occu-time-BM}
    \int_0^t \Pro \left( X(s) = 0\right) ds = \E\left( \alpha L^X(t,0) \right)= \frac{\sqrt{2}}{\sqrt{\pi}}\int_0^t \frac{1}{\sqrt{y}} \int_{\frac{t-y}{\alpha}}^\infty e^{-\frac{z^2}{2y}}dzdy.
  \end{equation}
\end{lem}
\begin{proof}
  First, notice that by Theorem \ref{thm:occupation-time}, $\alpha L^X(t,0)\leq t$ with probability 1. Therefore,
  \[\E\left(\alpha L^X(t,0) \right) = \int_{0}^{t} \Pro\left(\alpha L^X(t,0)>y\right)dy.\]
  Lemma \ref{lem:local-time-dist} implies that
  \[\E\left(\alpha L^X(t,0) \right) = \int_0^t \frac{\sqrt{2}}{\sqrt{\pi(t- y)}} \int_{y/\alpha}^\infty e^{-\frac{z^2}{2(t- y)}}dz dy. \]
  The result follows by switching the roles of $y$ and $(t-y)$.
\end{proof}

We mentioned in Section \ref{S:Occu} that delayed processes spend positive time at their delay point. The next proposition characterizes the probability that the delayed Brownian process $X(t)$ is at $0$. In other words, the distribution of $X$ contains a point mass at the delay point $x=0$.
\begin{prop}
  For any $t>0$,
  \begin{equation} \label{eq:Pro-X=0}
    \Pro(X(t) = 0) = 2 e^{\frac{2t}{\alpha^2}}\left(1 - \Phi\left(\frac{2\sqrt{t}}{\alpha} \right) \right).
  \end{equation}
  where
  \[\Phi(x) = \frac{1}{\sqrt{2\pi}} \int_{-\infty}^x e^{-\frac{z^2}{2}}dz\]
  is the cumulative distribution function of a standard Gaussian random variable.
\end{prop}
\begin{proof}
 The main observation is that $\Pro(X(t)=0)$ is the derivative of \eqref{eq:expected-occu-time-BM}. That is
 \[\Pro(X(t)=0) = \frac{d}{dt} \E(\alpha L^X(t,0)) = \frac{\sqrt{2}}{\sqrt{\pi t}} \int_0^\infty e^{-\frac{z^2}{2t}} dz - \frac{\sqrt{2}}{\sqrt{\pi}} \int_0^t \frac{1}{\alpha \sqrt{y}}e^{-\frac{(t-y)^2}{2\alpha^2y}}dy.\]
 The first integral is equal to $1$. To simplify the second integral, we observe that
 \[\frac{1}{\alpha \sqrt{2\pi y}}e^{-\frac{(t-y)^2}{2\alpha^2y}} = \frac{d}{dy} \left(-\Phi\left(\frac{t-y}{\alpha \sqrt{y}}\right) + e^{\frac{2t}{\alpha^2}}\Phi\left(\frac{t+y}{\alpha \sqrt{y}} \right) \right).\]
 The result follows.
\end{proof}

\begin{rem}
  Notice that when $t=0$, $\Pro(X(0)=0)=1$ and this agrees with formula \eqref{eq:Pro-X=0}. Moreover,
  \begin{equation*}
    \lim_{\alpha \to 0} \Pro(X(t)=0) = 0 \ \text{ and } \lim_{t \to +\infty} \Pro(X(t) = 0) = 0.
  \end{equation*}
  This agrees with what we would expect. As $\alpha \to 0$, $X(t)$ behaves more and more like Brownian motion. The long-time limit is a consequence of the fact that the process is diffusive.
\end{rem}

Now that we have calculated $\Pro(X(t)=0)$, we can characterize the distribution of $X(t)$.
\begin{prop}
  For fixed $t>0$, the characteristic function for $X(t)$ is
  \begin{equation*}
    \phi(\lambda,t) = \E\left(e^{i\lambda X(t)}\right) = e^{-\frac{\lambda^2 t}{2}} - \lambda^2 \int_0^t e^{-\frac{\lambda^2}{2}(t-s)} e^{\frac{2s}{\alpha^2}}\left(1 - \Phi\left(\frac{2\sqrt{s}}{\alpha} \right) \right)ds.
  \end{equation*}
\end{prop}
\begin{proof}
  By It\^{o} formula,
  \begin{equation*}
    \E\left(e^{i\lambda X(t)}\right) = 1 - \frac{\lambda^2}{2} \int_0^t \E\left(e^{i\lambda X(s)}\mathbbm{1}_{\{X(s)\not=0\}} \right)ds.
  \end{equation*}
  Notice that
  \[\E\left(e^{i\lambda X(s)}\mathbbm{1}_{\{X(s)\not=0\}} \right) = \E\left(e^{i\lambda X(s)}\right) - \Pro(X(s)=0).\]
  Therefore, by \eqref{eq:Pro-X=0}, $\phi(\lambda, t)$ solves the ODE
  \begin{equation*}
    \frac{\partial \phi}{\partial t}(\lambda,t) = -\frac{\lambda^2}{2} \left(\phi(\lambda,t) - 2e^{\frac{2t}{\alpha^2}}\left(1 - \Phi\left(\frac{2\sqrt{t}}{\alpha}\right) \right) \right).
  \end{equation*}

  The latter implies the statement of the proposition.
\end{proof}

One can recover the distribution of $X(t)$ from its characteristic function $\phi(\lambda, t)$ by Fourier inversion. The measure will have two parts, a point mass at $0$ of weight given by \eqref{eq:Pro-X=0}, and an absolutely continuous part with respect to Lebesgue measure, described by a density.

\end{appendices}



\end{document}